\documentclass[12pt]{article}

\usepackage{amsmath,amsthm,amssymb,amsfonts}
\usepackage[utf8]{inputenc}
\usepackage[mathscr]{euscript}
\usepackage{graphicx}
\usepackage{color}
\usepackage{epsfig}

\newtheorem{thm}{Theorem}
\newtheorem{lemma}[thm]{Lemma}
\newtheorem{prop}[thm]{Proposition}
\newtheorem{cor}[thm]{Corollary}
\newtheorem{conj}{Conjecture}
\newtheorem{question}{Question}

\theoremstyle{definition}
\newtheorem{defn}{Definition}

\theoremstyle{remark}

\newcommand{\ff}{\mathbb{F}}

\DeclareMathOperator{\rank}{rank}

\begin{document}

\title{Successful Pressing Sequences for a Bicolored Graph and Binary Matrices}

\author{Joshua Cooper and Jeffrey Davis\\
Department of Mathematics\\
University of South Carolina\\
Columbia, SC 29208}

\maketitle

\begin{abstract}
We apply matrix theory over \(\ff_2\) to understand the nature of so-called ``successful pressing sequences'' of black-and-white vertex-colored graphs.  These sequences arise in computational phylogenetics, where, by a celebrated result of Hannenhalli and Pevzner, the space of sortings-by-reversal of a signed permutation can be described by pressing sequences.  In particular, we offer several alternative linear-algebraic and graph-theoretic characterizations of successful pressing sequences, describe the relation between such sequences, and provide bounds on the number of them.  We also offer several open problems that arose as a result of the present work.
\end{abstract}

\noindent MSC classes:	05C50, 15B33, 92D15.

\noindent Keywords: Binary matrix, pressing sequence, adjacency matrix, matching, bicolored graph.

\section{Introduction}

In a now classical paper in bioinformatics \cite{Hannenhalli}, Hannenhalli and Pevzner showed that there is a polynomial time algorithm to sort signed permutations by reversals, i.e., turn any signed permutation into the identity by reversing subwords (and flipping their signs).  This has important implications for computational phylogenetics: when comparing the sequence of genes of two related species, the shortest length of a sequence of reversals that transforms one into the other is one prominent measure of the evolutionary distance of the associated organisms.  The authors' strategy, and one that was improved upon in later work (for example, \cite{Kaplan}), is to construct the so-called ``breakpoint graph'' for the permutation to be sorted, show that a certain operation on the breakpoint graph corresponds to reversals, and then use certain numerical invariants of subgraphs to guide the sequence of moves to the identity.

This framework is now a keystone of bioinformatics algorithms, but it leaves many questions unanswered.  In particular, the proposed methodologies generate just one successful sorting of the signed permutation under consideration, and it is understood that there are often many such minimum-length sorting sequences.  Since each is only representative of one {\em possible} evolutionary history, it would be valuable to be able to sample from all possible such sequences to obtain more sensitive statistical properties.  As of yet, there is no full understanding of the space of possible histories, so Markov Chain Monte Carlo methods are valuable for approximately uniform sampling.  Such approaches present their own problems, however: it is necessary to obtain a proof of connectivity of the underlying graph of the Markov Chain to know that it will eventually reach every vertex; and it is necessary to obtain bounds on the mixing time of the process to ensure that near-uniformity will be achieved in reasonable time.  Indeed, some researchers have investigated these very kinds of questions: see, for example, \cite{mcmc}.

In order to state our results and situate it in the above discussion, we need the following definitions.

\begin{defn}
A \textbf{bicolored} graph is a pair \((G,c)\) where \(G\) is a simple graph, and \(c : V(G) \rightarrow \{\textrm{black},\textrm{white}\}\) is a coloring of its vertices.  Write \(\overline{\textrm{black}} = \textrm{white}\) and \(\overline{\textrm{white}} = \textrm{black}\).
\end{defn}

Denote by \(V(G)\) the vertex set of a graph, \(E(G)\) its edge set, and \(G[S]\) the induced subgraph of a set \(S \subset V(G)\); let \(N(v) = N_G(v)\) denote the neighborhood of \(v \in V(G)\), i.e., \(\{w \in V(G) : \{v,w\} \in E(G)\}\), and \(N^\ast(v) = N^\ast_G(v)\) the closed neighborhood of \(v\), i.e., \(N^\ast_G(v) = N_G(v) \cup \{v\}\).

\begin{defn}
Consider a bicolored graph, \((G,c)\) with a black vertex \(v \in V(G)\).  ``Pressing \(v\)'' is the operation of transforming \((G,c)\) into \((G',c')\), a new bicolored graph in which \(G[N^\ast(v)]\) is complemented.  That is, \(V(G')=V(G)\),
\[
E(G') = E(G) \triangle \binom{N^\ast(v)}{2},
\]
(where ``\(\triangle\)'' denotes symmetric difference) and \(c'(w) = c(w)\) for \(w \not \in N^\ast(v)\) and \(c'(w) = \overline{c(w)}\) for \(w \in N^\ast(v)\).
\end{defn}

\begin{figure}[h]
\begin{center}
\begin{tabular}{c}
\includegraphics[scale=.5]{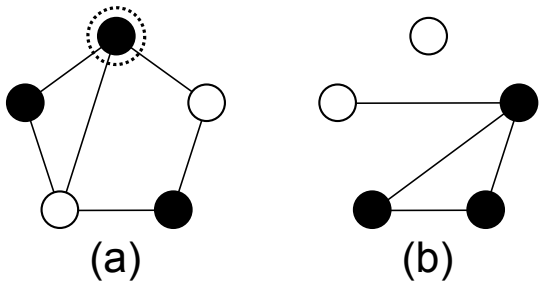}
\end{tabular}
\caption{The vertex enclosed by a dotted circle is pressed in 
graph (a) to obtain graph (b).}
\end{center}
\end{figure}

The ``pressing game'' (to use terminology from \cite{Bixby}) is played by pressing black vertices of \(G\) iteratively with the ultimate goal of transforming \(G\) into an all-white, empty graph.  Hannenhalli and Pevzner showed (\cite{Hannenhalli}) that ``successful'' sequences of presses in the breakpoint graph of a signed permutation, i.e., sequences that result in an all-white empty graph, correspond bijectively to minimum-length sequences of reversals that turn the permutation into the identity.  Therefore, sampling from successful pressing sequences is equivalent to sampling from the minimum length sequences of reversals that sort a signed permutation.  In \cite{Bixby}, the authors make the following ``Pressing Game Conjecture'':

\begin{conj} Every successful pressing sequence can be reached from every other one by a sequence of edits that involve at most four deletions or insertions.
\end{conj}

If successful pressing sequences are taken to be the vertices of a graph \(\Pi(G)\), and the edges correspond to edits of at most four deletions or insertions, then the Pressing Game Conjecture implies that \(\Pi(G)\) is connected.  Then a simple random walk converges to a uniform distribution on the set of all successful pressing sequences, and Markov Chain Monte Carlo can be used to analyze typical pressing sequences.  Bixby, Flint, and Mikl\'{o}s \cite{Bixby} proved the conjecture for paths.  Despite this, the current authors have doubts about the statement for general graphs.

In the present manuscript, we explore a few aspects of matrix theory over \(\ff_2\) so as to better understand the successful pressing sequences of a graph.  Among other results: in Corollary \ref{cor:rank_cor}, we show that the rank of the ``augmented adjacency matrix'' of a bicolored graph is the length of every successful pressing sequence of a graph; Theorems \ref{thm:main_thm} and \ref{thm:main_thm2} provide a substantial collection of equivalent characterizations of successful pressing sequences; Proposition \ref{prop:LPQU} gives a matrix-theoretic formulation of the relationship between successful pressing sequences; and Theorem \ref{thm:average} shows that the average number of successful pressing sequences of a random (full-rank) bicolored graph is large.  The final section contains several open problems concerning these sequences that arose in connection with the present work.

We note that some special cases of a few of our results are announced but left largely unproven in \cite{Hartman}; the authors refer to the matrix analogue of pressing as ``clicking'' and to the condition of the existence of a successful pressing sequence that consists of all vertices as ``tightness.''

\section{Preliminaries}

The following result appears (less explicitly) in \cite{Hannenhalli} and \cite{Bergeron}, but we include the proof for completeness.\footnote{Thanks to \'{E}va Czabarka for suggesting this vastly simplified version of the Hannenhalli-Pevzner argument.}

 \begin{prop}\label{prop:success_is_possible}
 Any graph with a black vertex in every component has a successful pressing sequence.
 \end{prop}
 
 \begin{proof}
It suffices to prove the statement for connected graphs. Let \(G\) be a connected graph and \(X\) be the set of black vertices with the fewest possible black neighbors. Choose some \(x \in X\) such that \(\deg(x)\) is maximal in \(X\).  When \(x\) is pressed, we obtain \(G'\).  We claim that each component of \(G'\) is either a white isolated vertex or has at least one black vertex.
 
Let \(N = N_G(x)\) be the set of neighbors of \(x \in G\), let \(P\) be the set of vertices in \(N\) that were white in \(G\), and let \(Q = N \setminus P\).  Note that \(G\) and \(G'\) are identical except on the induced subgraphs of \(N \cup \{x\}\).  Every vertex in \(V(G) - (N \cup \{x\})\) is in a component with a vertex of \(N\) (in both \(G\) and \(G'\)), so it suffices to show that each vertex of \(N \cup \{x\}\) is in a component of \(G'\) with a black vertex or is an isolated white vertex.  Furthermore, in \(G'\), \(x\) is isolated and white and the vertices of \(P\) are black, so we need only consider the elements of \(Q\).

Pick some \(z \in Q\). If \(z\) is adjacent to a black vertex outside of \(N \cup \{x\}\) or \(z\) is not adjacent in \(G\) to some vertex in \(P\), then in \(G'\), \(z\) is adjacent to a black vertex. Otherwise, in \(G\), \(z\) is adjacent to all vertices of \(P\) and its black neighbors are a subset of \(Q \cup \{x\} \setminus \{z\}\). By the choice of \(x\), this implies that the closed neighborhoods of \(x\) and \(z\) are the same in \(G\), which implies that \(z\) is a white isolated vertex in \(G'\).
\end{proof}

\begin{defn}
The \textbf{augmented adjacency matrix} \(A(G) \in \ff_2^{n \times n}\) of a bicolored graph \(G\) on \(n\) vertices, is the adjacency matrix of \(G\) where the entries along main diagonal correspond to the vertices \(G\) and are indexed by the color of the vertex; \(0\) if white or \(1\) if black.
\end{defn}
 
Given a bicolored graph \(G\), we can define a (loopy simple, uncolored) graph \(\hat{G}\) to be the graph on the same vertex set with the same edges, but with a loop at each black vertex (and none at white vertices).  A perfect matching in such a graph is a set of edges incident to every vertex exactly once, where a loop is considered to be incident to its vertex only once.  A special case of the following result (that of zero diagonal) appears in \cite{GHL}.

\begin{prop} \label{prop:matchings} The number of perfect matchings in the loopy graph \(\hat{G}\) corresponding to a bicolored graph \((G,c)\) is odd if and only if \(A(G)\) is invertible over \(\ff_2\).
\end{prop}

\begin{proof} It is well known that the permanent (which is equal to the determinant in characteristic 2) of \(A(G)\) is equal to 
\[
\sum_{\mathcal{C}} 2^{Z(\mathcal{C})},
\]
where \(\mathcal{C}\) ranges over all {\em vertex circuit covers}, i.e., families of circuits (closed walks) in which each vertex appears exactly once and \(Z(\mathcal{C})\) is the number of such circuits of length greater than two.  (See, for example, \cite{Galvin}.)  Therefore, over \(\ff_2\), the only terms which make a contribution to \(\det(A(G))\) are those in which there are no circuits of length more than two, i.e., every component is a loop or a single edge -- precisely the condition of being a perfect matching.  Since \(\det(A(G))=1\) if and only if \(A(G)\) is invertible, this is equivalent to there being an odd number of perfect matchings.
\end{proof}

\section{Matrix Theory}
Define the function \(f(M)\) on \(n \times n\) nonzero matrices over \(\ff_{2}\) as follows; the action of \(f\) will amount to a slight modification of Gaussian elimination, wherein row-permutations are prohibited and the pivot row is added to itself as well.  Let \(s\) denote the smallest row index of a left-most \(1\) in \(M\), that is, the positive integer for which there exists a \(t\) so that
\begin{enumerate}
\item \(M_{s,t} = 1\)
\item \(M_{s,j} = 0\) if \(j < t\)
\item If \(i < s\) and \(j <= t\), then \(M_{i,j} = 0\).
\end{enumerate}
Note that \(s\) and \(t\) are uniquely determined by \(U\) in satisfying the above requirements.  Let \(U = U(M)\) be the set of row indices which have a \(1\) in column \(t\), i.e.,
\begin{center}
\(U = \{i : M_{i,t} = 1\}\).
\end{center}
Let \(f(M)\) denote the \(n \times n\) matrix so that
\[
f(M)_{i,j} = \left \{ \begin{array}{ll}
M_{i,j} & \textrm{ if } i \not \in U \\
M_{i,j} + M_{s,j} & \textrm{ if } i \in U
\end{array} \right .
\]
Note that, for every matrix \(M\), there is a sequence of \(s\)'s and \(t\)'s that arise from the iterative application of \(f\) to \(M\).  That is, given \(M\), there is an increasing sequence \(s_{1},s_{2},...,s_{p}\) and increasing sequence \(t_{1},...,t_{p}\) which serve as the indices in the above definition of \(f(M)\), \(f(f(M))\), etc.  Indeed, it is easy to see that the sequence must eventually result in the all-zeroes matrix, so this process terminates at some finite \(p = p(M)\).

If, for each \(r \in [p]\), \(s_r = r\), we call \(M\) ``leading principally nonsingular (LPN)''.  If \(M\) is LPN, then the sequence \(M, f(M), f(f(M)), \ldots, f^{(p)}(M)\), \(p = \rank(M)\), is precisely the sequence of matrices one obtains by performing Gaussian elimination on \(M\), with the additional operation of adding pivot rows to themselves (thus replacing them with the zero vector).  Furthermore, this elimination does not involve row permutations.  Therefore, \(M\) is row-reducible to a diagonal matrix whose leading principle submatrix of size \(p\) is the identity matrix \(I_p\), and whose other entries are zero.

Note that, if \(M\) is symmetric, then \(f(M)\) is, as well.  Indeed,
\begin{align*}
f(M)_{i,j} &= M_{i,j} + M_{s,j} \cdot M_{i,s} && \textrm{by the definition of \(f\)} \\
&= M_{j,i} + M_{j,s} \cdot M_{s,i} && \textrm{by the symmetry of \(M\)} \\
&= M_{j,i} + M_{s,i} \cdot M_{j,s} & \\
&= f(M)_{j,i} && \textrm{by the definition of \(f\)}.
\end{align*}
Therefore, suppose that the \(M\) above is \(A(G)\) and is LPN.  Then it is straightforward to see that \(f(M)\) is in fact \(A(G^\prime)\) where \(G^\prime\) is obtained from \(G\) by pressing its lowest-indexed (black) vertex.  Since \(A(G)\) being the all-zeroes matrix is precisely the condition that \(G\) has no edges and all vertices are white, \(M = A(G)\) being LPN is equivalent to \(G\) having a successful pressing sequence consisting of the vertices indexing the first \(\rank(A(G))\) columns of \(M\) in increasing order.

We may conclude the following.

\begin{prop} \label{prop:row_reduction}
For a graph \(G\) with vertex set \([n]\), the following are equivalent:
\begin{enumerate}
\item Gaussian elimination applied to \(A(G)\) consists of row reduction applied, in consecutive order, to the first \(\rank(A(G))=k\) rows.
\item The first \(k\) leading principal minors of \(A(G)\) are nonzero, and the rest are zero.
\item \(1,2,\ldots,k\) is a successful pressing sequence for \(G\).
\end{enumerate}
\end{prop}

\begin{cor} \label{cor:rank_cor} The number of vertices in any successful pressing sequence for a graph \(G\) depends only on the graph, and is equal to \(\rank(A(G))\).
\end{cor}

The preceding result justifies the following definition.

\begin{defn}
The \textbf{pressing number} of a graph is the number of presses required to transform the graph into an all-white, empty graph.
\end{defn}

A matrix \(M\) is said to have an \(LU\)-decomposition if there exist a lower triangular matrix \(L\) and an upper triangular matrix \(U\) so that
\[
M = LU.
\]
Call a matrix \(M\) ``Cholesky'' (to borrow terminology from the theory of real/complex matrices; q.v.~\cite{Horn}) if there exists a lower-triangular \(L\) so that \(M = L L^T\); such a product is evidently a special type of \(LU\)-decomposition.  The following lemma is folkloric.

\begin{lemma} \label{lem:LU_is_unique} If \(M = LU\) and \(M\) is invertible, this decomposition is unique.
\end{lemma}
\begin{proof} Suppose \(M = LU = L'U'\).  Then \(L\), \(U\), \(L'\), and \(U'\) are invertible, so
\[
L'^{-1} L = U' U^{-1}.
\]
The left-hand side of this equation is lower-triangular and the right is upper-triangular, so they must both be diagonal.  However, since the only invertible diagonal matrix over \(\ff_2\) is the identity, \(L = L'\) and \(U = U'\).
\end{proof}

\begin{lemma} \label{lem:symmetric_and_LU} If a symmetric matrix \(M\) over \(\ff_2\) has an \(LU\)-decomposition, then it has a Cholesky decomposition \(\tilde{L} \tilde{L}^T\).
\end{lemma}
\begin{proof}
We proceed by induction.  The base case is trivial: \(M = [0]\) or \(M = [1]\).  Suppose \(M\) is \(n \times n\) and the statement is true for all \(1 \leq k < n\).  If the first row of \(M\) is all zero, then Gaussian elimination of \(M\) can proceed without row permutations if and only if this is true of \(M^\prime\), the matrix \(M\) with its first row and column removed.  Since the existence of an \(LU\)-decompostion is equivalent to Gaussian elimination proceeding without the necessity of row permutations, it follows that we may apply the inductive hypothesis to \(M^\prime\).  Suppose the first row of \(M\) is nonzero.  Then \(M\) can be written
\begin{align*}
M &= LU \\
& = \left [ \begin{array}{ccc} L_0 & O \\ A & B \end{array} \right ]
\left [ \begin{array}{ccc} U_0 & C \\ O & D \end{array} \right ] \\
& = \left [ \begin{array}{ccc} L_0U_0 & L_0C \\ AU_0 & AC+BD \end{array} \right ]
\end{align*}
where \(L_0\) are \(U_0\) are invertible leading principal submatrices, \(B\) is lower-triangular and \(D\) is upper-triangular -- because, otherwise, the first row or column of \(M\) would be zero.  Since \(M_0 = L_0U_0\) is nonsingular, this decomposition of \(M_0\) is unique by Lemma \ref{lem:LU_is_unique}.  Thus, \(M_0 = M_0^T = U_0^{T} L_0^{T}\) implies that \(U_0^T = L_0\). Since \(M\) is symmetric,
\[
L_0C = (AU_0)^T = U_0^T A^T = L_0 A^T,
\]
whence \(C = A^T\).  Therefore, if \(M_1 = AC + BD\) is the lower-right \((n-k)\times(n-k)\) principal submatrix of \(M\), we may rewrite
\[
M_1 - AA^T = BD.
\]
Since the left-hand side is symmetric, the right-hand is as well, and we may apply induction: \(M_1 - AA^T\) has an \(LU\)-decomposition of the form \(BD\), so it has an \(L_1L_1^T\)-decomposition as well.  Therefore, we may let
\[
\tilde{L} = \left [ \begin{array}{ccc} L_0 & O \\ A & L_1 \end{array} \right ]
\]
and conclude that \(M = \tilde{L} \tilde{L}^T\).
\end{proof}

\begin{thm}\label{thm:main_thm} Given a bicolored labeled graph \(G\) and integer \(k\), the following are equivalent:
\begin{enumerate}
\item The pressing number of \(G\) is \(k\).
\item \(A(G)\) has rank \(k\) and can be written
\[
A(G) = P^T L L^T P
\]
for some lower-triangular matrix \(L\) and permutation matrix \(P\).
\item \(rank(A(G)) = k\) and \(G\) has a black vertex in each component that is not an isolated vertex.
\item There is some permutation matrix \(P\) so that the \(j\)-th leading principal minor of \(P^T A(G) P\) is nonzero for \(j \in [k]\) and is zero for \(j > k\) if \(k < n\).
\item There is an ordering of the vertices \(v_1,\ldots,v_n\) of \(\hat{G}\) so that the induced subgraph \(\hat{G}[\{v_1,\ldots,v_j\}]\) has an odd number of perfect matchings for each \(j \in [k]\), and, for each \(j \in [n] \setminus [k]\), \(\hat{G}[\{v_1,\ldots,v_j\}]\) has an even number of perfect matchings.
\item \(A(G) = P^T L U P\) for some permutation matrix \(P\), lower triangular matrix \(L\), and upper triangular matrix \(U\), where \(\rank(LU) = k\).
\end{enumerate}
\end{thm}

We note that there are yet other equivalent conditions in terms of rank of submatrices which are given by Johnson and Okunev in \cite{Okunev}.

\begin{proof}
\(1 \Leftrightarrow 6\): This follows from Proposition \ref{prop:row_reduction}, since the existence of an \(LU\)-decomposition is equivalent to a matrix being row-reducible without performing row permutations; conjugation by \(P\) has the effect of placing rows and columns indexed by the pressing sequence in an initial position of the matrix.\\
\(2 \Leftrightarrow 6\): This is a consequence of Lemma \ref{lem:symmetric_and_LU}. \\ 
\(1 \Leftrightarrow 3\): The combination of Proposition \ref{prop:success_is_possible}, Proposition \ref{prop:row_reduction}, and Corollary \ref{cor:rank_cor} gives this equivalence.\\
\(1 \Leftrightarrow 4\): This follows from Proposition \ref{prop:row_reduction}.\\
\(4 \Leftrightarrow 5\): This is a consequence of Proposition \ref{prop:matchings}.\\
All conditions are therefore equivalent by transitivity.
\end{proof}

We may consider the special case when the labeling provides a successful pressing sequence.

\begin{thm}\label{thm:main_thm2} Given a bicolored labeled graph \(G\) on \([n]\), the following are equivalent:
\begin{enumerate}
\item The vertices of \(G\), in the usual order, are a successful pressing sequence.
\item \(A(G)\) can be written
\[
A(G) = L L^T
\]
for some invertible lower-triangular matrix \(L\).
\item Every leading principal minor of \(A(G)\) is nonzero for \(j \in [n]\).
\item The induced subgraph \(\hat{G}[\{1,\ldots,j\}]\) has an odd number of perfect matchings for each \(j \in [n]\).
\item \(A(G) = L U\) for some invertible lower triangular matrix \(L\) and invertible upper triangular matrix \(U\).
\end{enumerate}
\end{thm}
\begin{proof} All statements are simply specializations of those from the preceding Theorem, except for \(2\).  That \(L\) (and not just \(A(G)\)) can be taken to be full rank follows from the fact that \(\rank(L)=\rank(LL^T)\) if \(L\) is invertible.
\end{proof}

\section{Enumeration and Pressing Sequences}

\begin{prop} \label{prop:onebicoloringpersequence} Given a graph \(G\) with vertex set \([n]\) and a permutation \(\sigma\) of \([n]\), there is exactly one bicoloring of \(G\) for which \(\sigma\) is a valid pressing sequence.
\end{prop}
\begin{proof} We apply the part of Theorem \ref{thm:main_thm} that says that \(\sigma\) is a valid pressing sequence if and only if \(\hat{G}[\sigma(1),\ldots,\sigma(k)]\) has an odd number of perfect matchings for each \(k \in [n]\).  The proof proceeds by induction on \(k\).  If \(k=1\), clearly the number of perfect matchings is odd if and only if \(\hat{G}[\sigma(1)]\) has a loop, so \(\sigma(1)\) must be black in the bicoloring.  Suppose the statement is true for \(k < n\).  Let \(a\) denote the number of perfect matchings of \(\hat{G}[\sigma(1),\ldots,\sigma(k)]\), \(b\) the number of perfect matchings of \(\hat{G}[\sigma(1),\ldots,\sigma(k+1)]\), and \(c\) the number of perfect matchings of \(\hat{G}[\{\sigma(1),\ldots,\sigma(k)\} \setminus N(\sigma(k+1))]\).  Then \(b = c\) if \(\sigma(k+1)\) does not have a loop, while \(b = a + c\) if \(\sigma(k+1)\) does have a loop.  Since \(a\) is odd by the inductive hypothesis, exactly one of \(c\) or \(a+c\) is odd, so the color of \(\sigma(k+1)\) is uniquely determined.
\end{proof}

As with real/complex matrices, a matrix \(Q\) over \(\ff_2\) is said to be ``orthogonal'' if \(Q^T Q = I\).  The set of all \(n \times n \) orthogonal matrices over \(\ff_2\) is the ``orthogonal group'' and is denoted \(\mathscr{O}(n)\). 

\begin{prop} \label{prop:LPQU} Suppose the bicolored graph \(G\) with vertex set \([n]\) has the identity permutation as a pressing sequence, let \(A\) be the augmented adjacency matrix of \(G\), and let \(A = L L^T\) be the Cholesky decomposition guaranteed by Theorem \ref{thm:main_thm}.  Then \(\sigma\) is also a pressing sequence of \(G\) iff there exist an orthogonal matrix \(Q\) and an upper triangular matrix \(U\) so that
\[
L^T P^T = QU,
\]
where \(P\) is the permutation matrix encoding \(\sigma\).
\end{prop}
\begin{proof} If \(L^T P^T = QU\), then
\[
P A P^T = P L L^T P^T = U^T Q^T Q U = U^T U,
\]
a Cholesky decomposition for the matrix \(P A P^T\).  On the other hand, suppose \(\sigma\) is a pressing sequence for \(G\).  Then
\[
P A P^T = U^T U.
\]
for some \(U\), whence
\[
P L L^T P^T = U^T U.
\]
Let \(Q = L^T P^T U^{-1}\); then \(Q^T Q = I\), so \(Q\) is orthogonal, and
\[
L^T P^T = Q U.
\]
\end{proof}

It is worth remarking that one may take \(P\) to be any permutation matrix representing an automorphism of \(G\), \(Q = I\), and \(U = L^T\) to obtain a solution to \(L^T P^T = Q U\).  Indeed, acting on a bicolored graph by an automorphism fixes its pressing sequences.  Therefore, by the above proposition, one may view successful pressing sequences as a kind of \(\ff_2\)-relaxation of automorphisms.

Given a matrix \(B\), we define a new matrix \(\psi(B)\) as follows.  Let the columns of \(B\) be \(b_1,\ldots,b_n\) and the columns of \(\psi(B)\), \(b_1^\prime,\ldots,b_n^\prime\).  If \(b_1^\prime,\ldots,b_k^\prime\) have been defined for \(k < n\), we define
\[
b_{k+1}^\prime = b_{k+1} + \sum_{j=1}^k b_j^\prime (b_j^\prime \cdot b_{k+1}).
\]
Note that we can also define \(b_{k+1}^\prime\) by
\[
b_{k+1}^\prime = b_{k+1} + \sum_{\substack{j < k \\ b_j^\prime \cdot b_{k+1} = 1}} b_j^\prime.
\]

\begin{prop} Suppose the bicolored graph \(G\) with vertex set \([n]\) has the identity permutation as a pressing sequence, let \(A\) be the augmented adjacency matrix of \(G\), and let \(A = L L^T\) be the Cholesky decomposition guaranteed by Theorem \ref{thm:main_thm2}.  Let \(\sigma\) be a permutation of \([n]\) and \(P\) is the permutation matrix encoding \(\sigma\).  Then \(\sigma\) is a valid pressing sequence for \(G\) iff the all-ones vector \(\hat{1}\) is a (left) eigenvector of \(\psi(L^T P^T)\) iff \(\psi(L^T P^T)\) is orthogonal.
\end{prop}
\begin{proof}
Note that the computation of \(\psi(B)\) is precisely that of performing the Gram-Schmidt algorithm on the columns of \(B\), except that no ``normalization'' occurs, i.e., one does divide by the norm of the resulting vectors.  However, over \(\ff_2\), there are only two possible ``norms'': 0 and 1.  Therefore, if the norm of each of the \(b_j^\prime\) produced in the computation of \(\psi(B)\) is \(1\) for all \(j\), the columns of \(\psi(B)\) are the same as the output of Gram-Schmidt orthonormalization, whence we obtain a factorization of \(B\) of the form \(QU\) with \(Q\) orthogonal and \(U\) upper triangular.  (This is usually termed a ``QR-factorization''.)  The only failure mode of this computation is if some \(b_j\) is self-orthogonal.  Since self-orthogonality is equivalent to having an inner product of \(0\) with \(\hat{1}\), by Proposition \ref{prop:LPQU}, if \(\sigma\) is a successful pressing sequence, \(\psi(L^T P^T)\) is orthogonal and \(\hat{1}^T \psi(L^T P^T) = \hat{1}^T\); otherwise, \(\psi(L^T P^T)\) is not orthogonal and \(\hat{1}^T \psi(L^T P^T) \neq \hat{1}^T\).
\end{proof}

\begin{thm} \label{thm:average} For an (ordinary) graph \(G\) on \(n\) vertices, let \(\alpha_G\) denote the average number of length-\(n\) successful pressing sequences over all bicolorings of \(G\).  Then
\[
\alpha_G = \frac{n!}{2^n}.
\]
\end{thm}
\begin{proof} We construct a bipartite graph \(\Gamma\) as follows.  One partition class \(S\) consists of all permutations of \([n]\); the other partition class is \(C\), the set of all bicolorings of \(G\), which we assume has vertex set \([n]\).  We place an edge between a permutation \(\sigma\) and a bicolored graph \(G\) iff \(\sigma\) is a successful pressing sequence for \(G\).  By Proposition \ref{prop:onebicoloringpersequence}, the \(\Gamma\)-degree of each vertex in \(S\) is \(1\), so the number of edges in \(\Gamma\) is \(n!\).  On the other hand, the number of edges incident to a bicolored graph is its number of length-\(n\) successful pressing sequences.  Therefore, the sum of all degrees in \(C\) is \(\alpha_G 2^n\).  We may conclude that
\[
\alpha_G = \frac{n!}{2^n}.
\]
\end{proof}

Note that, since a particular labeling of a graph admits precisely one bicoloring so that \(\rank(A(G))=n\), the probability that a symmetric matrix (i.e., the adjacency matrix of a bicolored graph) is Cholesky \(2^{-n}\).

\section{Conclusion}

We present a few open questions on the subject of pressing sequences in addition to the Pressing Game Conjecture (discussed in the introduction).

\begin{question} How hard is it in general to compute the number of successful pressing sequences of a given bicolored graph?
\end{question}

By the remarks following Proposition \ref{prop:onebicoloringpersequence}, it is perhaps the case that this enumeration problem is GI-complete, i.e., the same difficulty as certifying graph isomorphism and counting automorphisms.  Alternatively, the connection with counting perfect matchings suggests it might be \#P-hard.  Given that we do not know the complexity of counting pressing sequences exactly, perhaps the approximation problem is easier:

\begin{question} Is there a polynomial time algorithm for approximating within a small factor the number of successful pressing sequences of a given bicolored graph?
\end{question}

In studying some of these questions, the authors found a substantial, though small, number of nonisomorphic graphs which have exactly one pressing sequence -- graphs we term ``uniquely pressable''.  However, we lack a characterization of these graphs.

\begin{question} Describe the uniquely pressable graphs.
\end{question}

\section{Acknowledgments}

Thank you to \'{E}va Czabarka and Kevin Costello for helpful and insightful discussions.  Thank you also to the South Carolina Honors College at the University of South Carolina for supporting the present work with their Science Undergraduate Research Fellowship.

%\bibliographystyle{plain}
%\bibliography{ref}

\end{document}